\documentclass[letterpaper, reqno, 11pt]{amsart}
\usepackage[english]{babel}
\usepackage{graphicx,color}
\usepackage[all]{xy}
\usepackage{amsfonts,dsfont,mathrsfs}
\usepackage{amsmath, amsthm, amssymb}
\usepackage{enumerate, url}

\DeclareMathOperator{\diam}{diam}

\DeclareMathOperator{\vol}{vol}
\DeclareMathOperator{\Area}{Area}

\DeclareMathOperator{\Con}{Con}

\numberwithin{equation}{section}

\newtheorem{thmnr}{Theorem}[section]

\newtheorem{propnr}[thmnr]{Proposition}

\newtheorem{lemnr}[thmnr]{Lemma}

\newtheorem{cornr}[thmnr]{Corollary}
\theoremstyle{definition}

\newtheorem{dfnnr}[thmnr]{Definition}

\newtheorem{rmknr}[thmnr]{Remark}

\newtheorem*{prob}{Problem}

\newtheorem*{conj}{Conjecture}

\thanks{The first named author is partially supported by the NSF grant DMS-1207782.  
The second named author is partially supported by the NSF grant DMS-1207655.  
The third named author is partially supported by the NSF grant DMS-1406209.}

\begin{document}

\title{Quasicircle boundaries and exotic almost-isometries}
\author[J.-F. Lafont]{Jean-Fran\c{c}ois Lafont}
\author[B. Schmidt]{Benjamin Schmidt}
\author[W. Van Limbeek]{Wouter Van Limbeek}

\date{\today}

\begin{abstract}
We consider properly discontinuous, isometric, convex cocompact actions of surface groups $\Gamma$ on a 
CAT(-1) space $X$. We show that the limit set of such an action, equipped with the canonical visual metric, is a 
(weak) quasicircle in the sense of Falconer and Marsh.  It follows that the visual metrics on such limit sets are 
classified, up to bi-Lipschitz equivalence, by their Hausdorff dimension. This result applies in particular to 
boundaries at infinity of the universal cover of a locally CAT(-1) surface. We show that any 
two periodic CAT(-1) metrics on $\mathbb H^2$ can be scaled so as to be almost-isometric (though in 
general, no equivariant almost-isometry exists). We also construct, on each higher genus surface, $k$-dimensional
families of equal area Riemannian metrics, with the property that their lifts to the universal covers are 
pairwise almost-isometric 
but are {\bf not} isometric to each other. Finally, we exhibit a gap phenomenon for the optimal multiplicative 
constant for a quasi-isometry between periodic CAT(-1) metrics on $\mathbb H^2$.
\end{abstract}

\maketitle


\section{Introduction}

Consider a Riemannian manifold $(M,g)$ with curvature $\leq -1$ (or more generally, 
equipped with a locally CAT(-1) metric).  The fundamental group $\pi_1(M)$ acts via deck transformations on 
the universal cover $X$, and the metric $g$ lifts to a $\pi_1(M)$-invariant metric $\tilde g$ on $X$. An important
theme has been the study of the dynamics of $\pi_1(M)$ on the boundary at infinity $\partial X$. If one equips the
boundary at infinity with the canonical visual metric $d_\partial$ (see Definition \ref{dfn:vismetric}), then it is well-known 
that the boundary at infinity exhibits some fractal-like behavior. More generally this phenomenon occurs if $\pi_1(M)$ acts properly discontinuously and convex cocompactly on a CAT(-1) space $X$. Our first result follows this general philosophy:

\begin{thmnr} Let $\Gamma$ be a surface group, and let $(X,d)$ be any proper CAT(-1) space on which $\Gamma$ acts isometrically, properly discontinuously, and convex cocompactly. Let $\Lambda$ be the limit set of the $\Gamma$-action on $X$ and let $d_\partial$ denote the canonical visual metric on $\Lambda$. Then $(\Lambda, d_\partial)$ is a (weak) quasicircle in the sense of Falconer-Marsh. 
\label{thm:main}
\end{thmnr}

As mentioned above, this theorem applies in particular if $(X,d)$ is the universal cover of a closed surface equipped with a locally CAT(-1) metric and $\Gamma$ acts by deck transformations. The results below are obtained from this special case of Theorem \ref{thm:main}.

\begin{cornr} Let $(M_1, d_1), (M_2, d_2)$ be any pair of closed surfaces equipped with locally CAT(-1) metrics,
and let $(X_i, \tilde d_i)$ be their universal covers. Then one can find real numbers $0< \lambda_i\leq1$, with 
$\max\{\lambda_1, \lambda_2\}=1$, having the property
that $(X_1,\lambda_1 \tilde d_1)$ is almost-isometric to $(X_2, \lambda_2 \tilde d_2)$.
\label{cor:cat-setting}\end{cornr}

Recall that an {\it almost-isometry} is a quasi-isometry with multiplicative constant $=1$. If $M_1, M_2$ are locally CAT(-1) manifolds, then the existence of an almost-isometry $X_1\rightarrow X_2$ forces strong constraints on the geometry of $M_1$ and $M_2$.  In many cases, this forces the universal covers to be isometric (see \cite{KaLaSc} for more information). On the other hand, it follows immediately from Corollary \ref{cor:cat-setting} that there are examples of almost-isometric universal covers that are not isometric -- take for example $d_1$ to be a Riemannian metric and $d_2$ to be a non-Riemannian metric. 

\begin{cornr} Let $M$ be a closed surface and let $g_1, g_2$ be two Riemannian metrics on $M$
with curvatures $\leq -1$. Equip $\partial X$ with the corresponding canonical visual metrics $\rho_1$ and $\rho_2$. 
Then the following
three statements are equivalent:
\begin{enumerate}
\item The topological entropies of the two geodesic flows on $T^1M$are equal.
\item The boundaries $(\partial X, \rho_1)$ and $(\partial X, \rho_2)$ are bi-Lipschitz equivalent.
\item The universal covers $(X, \tilde g_1)$ and $(X, \tilde g_2)$ are almost-isometric.
\end{enumerate}
\label{cor:lipequiv}\end{cornr}

So, after possibly scaling one of the metrics, we can ensure that the universal covers of any two Riemannian surfaces
are almost-isometric. Of course, when scaling, we also change the geometry of the metric, e.g. the area. Our next result
shows that one can arrange for examples with almost-isometric universal covers, while still keeping control of the area
of the surface.

\begin{thmnr} Let $M$ be a closed surface of genus $\geq 2$, and $k\geq 1$ an integer. One can find a $k$-dimensional family
$\mathcal F_k$ of Riemannian metrics on M, all of curvature $\leq -1$, with the following property. If $g, h$ are any two distinct metrics in $\mathcal F_k$, then
\begin{itemize}
\item $Area(M, g) = Area(M, h)$.
\item the lifted metrics $\tilde g, \tilde h$ on the universal cover $X$ are almost-isometric.
\item the lifted metrics $\tilde g, \tilde h$ on the universal cover $X$ are {\bf not} isometric.
\end{itemize}
\label{thm:main2}
\end{thmnr}

In the above Theorem \ref{thm:main2}, we think of the almost isometries between the lifted metrics 
on $X$ as being \textit{exotic} since they cannot be realized equivariantly with respect to the two natural $\pi_1(M)$-actions 
on $X$.   Indeed, the existence of a  $\pi_1(M)$-equivariant almost-isometry between the two 
lifted metrics on $X$ implies 
that the two metrics on $M$ have equal marked length spectra (see \cite{KaLaSc}, for example), and are therefore isometric 
by \cite{Cr,Ot}. 

As a final application, we exhibit a {\it gap phenomenon} for the optimal multiplicative constant for quasi-isometries between
certain periodic metrics.

\begin{cornr}
Let $(M_1, d_1), (M_2, d_2)$ be any pair of closed surfaces equipped with locally CAT(-1) metrics,
and assume that their universal covers $(X_i, \tilde d_i)$ are {\bf not} almost-isometric. Then there exists
a constant $\epsilon > 0$ with the property that {\it any} $(C, K)$-quasi-isometry from $(X_1, d_1)$ to $(X_2, d_2)$
must satisfy $C \geq 1+\epsilon$.
\label{cor:gap}
\end{cornr}



The layout of this paper is as follows. In Section \ref{sec:defn}, we review the basic definitions, and summarize
the results from the literature that we will need. We also show how to deduce Corollaries \ref{cor:cat-setting}, 
\ref{cor:lipequiv}, and \ref{cor:gap} from Theorem \ref{thm:main}. In Section \ref{sec:main}, we give a proof of 
Theorem \ref{thm:main}. Section \ref{sec:main2} is devoted to the proof of Theorem \ref{thm:main2}. Finally, 
we provide some concluding remarks in Section \ref{sec:conclusion}.

\subsection*{Acknowledgments:} We are pleased to thank Ralf Spatzier for helpful discussions. We would also like to 
thank Xiangdong Xie for informing us of Bonk and Schramm's work. Part of this work was completed during a collaborative
visit of the third author to Ohio State University (OSU), which was partially funded by the Mathematics Research Institute at OSU.


\section{Background material}\label{sec:defn}

In this section, we review some basic definitions we will need, and also provide descriptions
of some results we will need in our proofs.

\subsection{Convex cocompact actions} \label{sec:convcocmpct} 
We briefly summarize the statements we need concerning convex cocompact 
actions, and refer the reader to \cite[Section 1.8]{Bou} for more details. Given a properly
discontinuous isometric action of $\Gamma$ on a proper CAT(-1) space $X$, we have an associated {\it limit set} in 
the boundary at infinity $\partial X$. This set $\Lambda_\Gamma$ is obtained by taking the closure $\overline{\Gamma \cdot p}$ 
of the $\Gamma$-orbit of a point $p\in X$ inside the compactification $\overline{X}:= X \cup \partial X$, and setting
$\Lambda_\Gamma:= \overline{\Gamma \cdot p} \cap \partial X$. 

\begin{dfnnr}
The $\Gamma$-action on $X$ is {\it convex cocompact} if it satisfies any of the following equivalent conditions:
\begin{enumerate}
\item the map $\Phi: \Gamma \rightarrow X$ given by $\Phi(g) := g(x)$ is quasi-isometric.
\item the orbit of a point $\Gamma\cdot p$ is a quasi-convex subset of $X$ (i.e. every geodesic joining a pair of points in 
$\Gamma\cdot p$ lies within a uniform neighborhood of $\Gamma\cdot p$).
\item the action of $\Gamma$ on the convex hull of its limit set $\textrm{Co}(\Lambda_\Gamma)$ is cocompact.

\end{enumerate}
\end{dfnnr}

The equivalence of statements (1) and (2) can be found in \cite[Corollary 1.8.4]{Bou}, while the equivalence of (2) and
(3) is shown in \cite[Proposition 1.8.6]{Bou}. It follows from these conditions that $\Gamma$ is $\delta$-hyperbolic, and therefore we can compare the boundary $\partial X$ with the Gromov boundary $\partial \Gamma$. For our purposes, an important consequence of the action being convex cocompact is that the limit set $\Lambda_\Gamma$ is homeomorphic to $\partial\Gamma$. So in the special case where $\Gamma$ is a surface group, the limit
set $\Lambda_\Gamma$ is homeomorphic to $S^1$.

\subsection{Metrics on the boundary}
We refer the reader to \cite{Bou} for more details concerning this subsection.

Let $X$ be a CAT(-1) space with boundary at infinity $\partial X$.  Fix a basepoint $w \in X$.  The Gromov product 
$(\cdot \vert \cdot)_w:X \times X \rightarrow \mathbb{R}$ is defined by 
$$(p\vert q)_w:=\frac{1}{2}\left(d(w, p) + d(w, q) - d(p, q) \right)$$ 
for each $x,y \in X$, and extends to $\partial X \times \partial X$ by 
$$(x,y)_w:=\lim_{n \rightarrow \infty} (x_n \vert y_n)_w$$
where $\{x_n\}$ and $\{y_n\}$ are sequences in $X$ converging to $x$ and $y$.  
Gromov products induce a (family of) canonical visual metric(s) on the boundary $\partial X$ defined as follows.
\begin{dfnnr} Fix a basepoint $w\in X$.  Then the metric $d_w$ is defined by
$$d_w(x, y) := e^{-(x|y)_w}$$
This gives a family of bi-Lipschitz equivalent metrics, obtained by varying the choice of the basepoint $w$. By an abuse
of language, we will refer to this bi-Lipschitz class of metrics on $\partial X$ as {\it the canonical visual metric}.
\label{dfn:vismetric}
\end{dfnnr}


\begin{rmknr} 
If we let $\gamma_{xy}$ denote the geodesic joining $x$ and $y$, there is a universal constant $C$ with the property that,
for all pairs of points $x, y\in \partial X$, 
$$\left| (x|y)_w - d(w, \gamma_{xy})\right| < C.$$
It follows that the distance function $d_w$ is bi-Lipschitz equivalent to the function $\hat d_w$ defined via
$$\hat d_w(x, y) := e^{-d(w, \gamma_{xy})},$$
(even though $\hat d_w$ might not define a metric). 
\label{rmk:alternate}
\end{rmknr}

Let $\Gamma$ be a properly discontinuous group of isometries of $X$ acting convex cocompactly. Let $\Lambda\subseteq \partial X$ be 
the limit set of $\partial X$, $s=\textrm{Hdim}(\partial X)$ be the Hausdorff dimension of the canonical visual metric on 
$\Lambda$, and $H^{s}$ be the corresponding Hausdorff measure.  Then $H^{s}$ is a finite measure, and fully supported 
on $\Lambda$. The Hausdorff dimension and measure can be estimated
using the following result of Bourdon \cite[Theorem 2.7.5]{Bou}.

\begin{thmnr}[Bourdon]\text{}

\begin{enumerate}[(i)]
\item $s=\displaystyle\varlimsup_{n \to \infty} \frac{1}{n} \log \#\{\gamma \in \Gamma\, \vert\, d(w,\gamma w)\leq n\}$\\

\item There exists a constant $C \geq 1$ such that for each metric ball $B(x,r)$ in $\partial X$ (with center $x$
and radius $r$),
$$C^{-1}r^{s}\leq H^{s}\left(B(x,r)\right)\leq Cr^s.$$

\end{enumerate}
\label{thm:bourdon}\end{thmnr}

\begin{rmknr}

If $(M,g)$ is a closed Riemannian manifold with sectional curvatures $\leq -1$, 
then its universal Riemannian covering $X$ is a CAT(-1) space equipped with a geometric action of 
$\Gamma=\pi_1(M)$ by deck transformations.  In this case, $s=h(g)$, where the latter denotes the 
topological entropy of the geodesic flow on the unit tangent bundle $T^{1}(M)$.

To see this, let $W$ denote a bounded fundamental domain for the $\Gamma$-action, $V=\vol(W)$, 
and $D=\diam(W)$. The following basic estimates
$$ \vol(B(w,n))\leq V\cdot \#\{\gamma \in \Gamma\, \vert\, d(w,\gamma w)\leq n+D\}$$ and 
$$V\cdot \#\{\gamma \in \Gamma\, \vert\, d(w,\gamma w)\leq n\}\leq \vol(B(w, n+D))$$
imply that the Hausdorff dimension $s$ and the volume growth entropy $$h_{vol}(g):=\lim_{r \to \infty} \frac{1}{r} \log \vol(B(w,r))$$ coincide.  Then by Manning's theorem \cite{Ma}, $s=h_{\vol}(g)=h(g)$. 
\label{rmk:entropy}\end{rmknr}

\subsection{(Weak) Quasicircles according to Falconer-Marsh}

Next let us briefly review some notions and results of Falconer and Marsh \cite{falconer-marsh}.

\begin{dfnnr}[Falconer-Marsh] A metric space $(C,d)$ is a \emph{quasicircle} if
	\begin{enumerate}[(i)]
		\item $C$ is homeomorphic to $S^1$,
		\item (expanding similarities) There exist $a,b,r_0>0$ with the following property. For any $r<r_0$ and $N\subseteq C$ 			with $\diam(N)=r$, there exists an expanding map $f:N\rightarrow C$ with expansion coefficient between $\frac{a}{r}$ 			and $\frac{b}{r}$, i.e. for all distinct $x, y\in N$, we have the estimate
		$$\frac{a}{r}\leq  \frac{d(f(x), f(y))}{d(x,y)} \leq \frac{b}{r}$$
		Note that $a, b$ are independent of the size of $N$, and of the choice of points in $N$.
		\item (contracting similarities) There exist $c, r_1>0$ with the following property. For any $r<r_1$ and ball $B\subseteq 			C$ with radius $r$, there exists a map $f:C\rightarrow B\cap C$ contracting no more than a factor $cr$.
	\end{enumerate}
\label{dfn:qcircle}
\end{dfnnr}

Let $s=\textrm{Hdim}(C)$ be the Hausdorff dimension of $C$, and $H^s$ the corresponding Hausdorff measure. The following alternate property to (iii) is implied by (ii) and (iii):

\begin{enumerate}
		\item[(iiia)] For any open $U\subseteq C$ one has 
		$0<H^s(U)<\infty$, and $H^s(U)\rightarrow 0$ as $\diam(U)\rightarrow 0$.
	\end{enumerate}

The main result of Falconer-Marsh \cite{falconer-marsh} is that two quasicircles $C_1$ and $C_2$ are bi-Lipschitz equivalent if and only if their Hausdorff dimensions are equal.  While stated this way, their proof only uses conditions (i), (ii), and (iiia).  For this reason, we will say that a metric space $(C,d)$ is a \textit{weak quasicircle} when conditions (i), (ii), and (iiia) are satisfied.




\begin{thmnr}[Falconer-Marsh] Two weak quasicircles $C_1$ and $C_2$ are bi-Lipschitz equivalent if and only if their Hausdorff dimensions are equal.
\label{thm:falconermarsh}\end{thmnr}





\subsection{Almost-isometries and the work of Bonk-Schramm}

Now let us recall some results of Bonk and Schramm \cite{BS} that we will need.

\begin{dfnnr}
A map between $f: X\rightarrow Y$ between metric spaces $(X, d_X)$ and $(Y, d_Y)$ is \emph{quasi-isometric} if there exists
constants $C, K$ such that 
$$ \frac{1}{C}d_X(x,y) -K \leq d_Y\left(f(x), f(y)\right) \leq C d_X(x, y) + K.$$
A map $f: X\rightarrow Y$ is {\it coarsely onto} if $Y$ lies in a bounded neighborhood of $f(X)$.
If the quasi-isometric map $f$ is coarsely onto, then we call it a {\it quasi-isometry}, and we say that
$X, Y$ are {\it quasi-isometric}. A map is almost-isometric if it is quasi-isometric with multiplicative constant $C=1$. An
almost-isometric map $f: X\rightarrow Y$ which is coarsely onto is called an {\it almost-isometry}, in which case we say that 
$X,Y$ are {\it almost-isometric}.
\end{dfnnr}

Special cases of the results in \cite{BS} relate the existence of almost-isometries with metric properties of $\partial X$, as follows.

\begin{thmnr}[Bonk-Schramm] 
Let $X, Y$ be a pair of CAT(-1) spaces. Then
$X, Y$ are almost-isometric if and only if the canonical visual metrics on the boundaries $\partial X$, $\partial Y$ are
bi-Lipschitz homeomorphic to each other.
\label{thm:bongschramm}
\end{thmnr}

The fact that an almost-isometry between $X$ and $Y$ induces a bi-Lipschitz homeomorphism between the 
boundaries $\partial X$ and $\partial Y$ appears in \cite[proof of Theorem 6.5]{BS} -- where it should be noted 
that, in the notation of their proof, our more restrictive context corresponds to $\epsilon=\epsilon'=1$ and $\lambda=1$. 
As for the converse, the interested reader should consult \cite[Theorem 7.4]{BS} to see that a bi-Lipschitz map between 
boundaries $\partial X$ and $\partial Y$ induces an almost-isometry between the metric spaces $\Con(X)$ and $\Con(Y)$. 
The comment following \cite[Theorem 8.2]{BS} applies in our context where $a=1$, so that $\Con(X)$ and $\Con(Y)$ are 
almost isometric to $X$ and $Y$, respectively, whence $X$ and $Y$ are almost-isometric.

\subsection{Proof of Corollaries \ref{cor:cat-setting}, \ref{cor:lipequiv}, and \ref{cor:gap}}

The proof of all three corollaries are now completely straightforward.

\begin{proof}[Proof of Corollary \ref{cor:lipequiv}]
Theorem \ref{thm:main} gives us that 
$(\partial X, \rho_1)$ and $(\partial X, \rho_2)$ are weak quasicircles. Then Falconer and Marsh's Theorem \ref{thm:falconermarsh} and Remark \ref{rmk:entropy} gives the equivalence of statements (1) and (2), while Bonk and Schramm's Theorem \ref{thm:bongschramm} gives the
equivalence of statements (2) and (3).
\end{proof}

\begin{proof}[Proof of Corollary \ref{cor:cat-setting}]
When two metrics on $X$ are related by a scale factor $\lambda$, it easily follows from the formula for the canonical visual 
metric that the Hausdorff dimension of the boundary at infinity scales by $1/\lambda$. Corollary \ref{cor:cat-setting} immediately
follows by combining our Theorem \ref{thm:main}, Falconer and Marsh's Theorem \ref{thm:falconermarsh}, and Bonk and Schramm's Theorem \ref{thm:bongschramm}.
\end{proof}

\begin{proof}[Proof of Corollary \ref{cor:gap}]
Combining our Theorem  \ref{thm:main}, Falconer and Marsh's Theorem \ref{thm:falconermarsh}, and Bonk 
and Schramm's Theorem \ref{thm:bongschramm}, we see that the boundaries at infinity $(\partial X_i, \rho_i)$ must have 
distinct Hausdorff dimensions. Without loss of generality, let us assume 
$\textrm{Hdim}(\partial X_1, \rho_1) < \textrm{Hdim}(\partial X_2, \rho_2)$. 
If $\phi: X_1\rightarrow X_2$ is a $(C, K)$-quasi-isometry, we
want to obtain a lower bound on $C$. But the quasi-isometry induces a homeomorphism $\partial \phi : \partial X_1 \rightarrow
\partial X_2$ which, from the definition of the canonical visual metrics, has the property that
$$\rho_2\left(\partial \phi(x), \partial \phi(y) \right) \leq e^K \cdot \rho_1(x, y) ^C$$
for all $x, y\in \partial X_1$. It easily follows that $\textrm{Hdim}(\partial X_2, \rho_2)\leq C\cdot \textrm{Hdim}(\partial X_1, \rho_1)$, 
giving us the desired inequality
$$1 < \frac{\textrm{Hdim}(\partial X_2, \rho_2)}{\textrm{Hdim}(\partial X_1, \rho_1)} \leq C.$$
\end{proof}


\section{Surface boundaries are (weak) quasicircles} \label{sec:main}

This section is devoted to the proof of Theorem \ref{thm:main}. Looking at Definition \ref{dfn:qcircle}, property (i) is 
obvious (see the discussion in Section \ref{sec:convcocmpct}), while property (iiia) is an immediate consequence of Theorem \ref{thm:bourdon}.
It remains to establish property (ii).  

Let $\Gamma$ be a surface group, and let $(X,d)$ be a CAT(-1) space. Suppose $\Gamma$ acts on $X$ isometrically, properly discontinuously, and convex cocompactly. Let $\Lambda\subseteq \partial X$ be the limit set of $\Gamma$ and write $Y:=\textrm{Co}(\Lambda)$ for the convex hull of $\Lambda$. 

\begin{lemnr} There exists $K>0$ such that the following holds. Let $\eta$ be any geodesic in $X$ with endpoints in $\Lambda$ (and hence $\eta \subset Y$) and let $N_K(\eta)$ be the $K$-neighborhood of $\eta$. Then $N_K(\eta)\cap Y$ separates $Y$.
\label{lem:separation}
\end{lemnr}
\begin{proof} Let $\Sigma$ be a closed hyperbolic surface with fundamental group $\Gamma$ and fix a homotopy equivalence $f:Y\slash\Gamma\rightarrow \Sigma$. Lift $f$ to a map on universal covers $\tilde{f}:Y\rightarrow \tilde{\Sigma}$. Then $\tilde{f}$ is a $(C,L)$-quasi-isometry for some $C\geq 1$ and $L>0$. Further we can choose $D>0$ such that a $(C,L)$-quasigeodesic in $\tilde{\Sigma}$ is Hausdorff distance at most $D$ from a geodesic. We will prove the lemma with $K:=2CD+L$. 

To see this, let $\eta$ be a geodesic in $Y$. Then $\tilde{f}\circ\eta$ is a $(C,L)$-quasigeodesic and hence is contained in the $D$-neighborhood of a geodesic $\gamma$. Since $\gamma$ separates $\tilde{\Sigma}$, it follows that $\tilde{f}^{-1}(N_D(\gamma))$ separates $Y$. 
Now let $x\in \tilde{f}^{-1}(N_D(\gamma))$. Then we have
	$$d(x,\eta)\leq Cd(\tilde{f}(x),\tilde{f}\circ\eta)+L\leq 2CD+L,$$
which is the desired bound.\end{proof}

Note that for $K$ as in Lemma \ref{lem:separation}, $Y\backslash N_K(\eta)$ will consist of two unbounded components because $\partial N_K(\eta)=\partial \eta$ consists of two distinct points in $\Lambda\cong S^1$, so that $\Lambda\backslash \partial N_K(\eta)$ consists of two components. 

\begin{propnr}[Definite expansion] Fix a bounded fundamental domain $W$ for the $\Gamma$-action on $Y$ and fix $w \in W.$  Further choose $K>0$ as in Lemma \ref{lem:separation}. Then there exists $A>0$ with the following property.
\begin{itemize}
	\item Let $\eta$ be a geodesic in $Y$ such that $w\notin N_K(\eta)$,
	\item set $R:=d(w,\eta)$ and let $p$ be the projection of $w$ onto $\eta$,
	\item let $\gamma\in\Gamma$ be such that $\gamma p\in W$,
 	\item let $\xi$ be any geodesic in the unbounded component of $Y\backslash N_K(\eta)$ not containing $w$. 
\end{itemize}
Then
	$$R-A\leq d(w,\xi)-d(w,\gamma\cdot \xi)\leq R+A.$$
\label{lem:expansion}
\end{propnr}

\begin{proof} 
 As $X$ is CAT(-1), $X$ is also Gromov hyperbolic.  Hence, there is a $\delta>0$ with the property that for each geodesic triangle $\Delta(abc)$, the side $[a,b]$ is contained in the $\delta$-neighborhood of the union of the remaining sides $[a,c]\cup [b,c]$.  Let $D=\diam(W)$.  The estimates that follow will show the Lemma holds when $A=4D+4K+12\delta$.

Let $q$ (resp. $q'$) be the projection of $w$ onto $\xi$ (resp. $\gamma \xi$), and let $p'$ be the projection of $w$ onto $\gamma \eta$.  Then since $w ,\gamma p \in W$,
\begin{equation}\label{1}
d(w,p')=d(w, \gamma \eta) \leq d(w, \gamma p)\leq D 
\end{equation}
and
\begin{equation}\label{2}
d(\gamma p, p') \leq d(\gamma p,w) + d(w,p') \leq 2D.
\end{equation}
Now we claim that
\begin{equation}\label{3}
d(p,[w,q])< 3\delta+K
\end{equation}
and 
\begin{equation}\label{4}
d(p',[w,q'])<3\delta+K.
\end{equation}
To prove (\ref{3}), choose $x\in [w,q]\cap N_K(\eta)$ (note that $[w,q]\cap N_K(\eta)$ is nonempty because the endpoints of $\xi$ lie in the component of $Y\backslash N_K(\eta)$ that does not contain $w$). Let $u$ be the projection of $x$ onto $\eta$. If $d(p,u)<2\delta$ then we are done (since $d(x,u)\leq K$), so let us assume that $d(p,u)\geq 2\delta$. Consider the geodesic triangle $\Delta(wpu)$, and let $y$ be the point on $[p,u]$ at distance $2\delta$ from $p$. By the definition of $\delta$, we know that $d(y,[w,p]\cup[w,u])<\delta$. However, since $d(y,p)>\delta$ and $p$ is the projection of $y$ onto $[w,p]$, we have $d(y,[w,p])>\delta$. Therefore we must have $d(y,[w,u])<\delta$. Since we also know $d(y,p)=2\delta$, we conclude that $d(p,[w,u])<3\delta$.  Using convexity of distance functions, we have that $(u,x)$ maximizes the function $d:[w,u]\times[w,x]\rightarrow \mathbb{R}$.  As $[w,x]\subset [w,q]$, $$d(p,[w,q])\leq d(p,[w,x])\leq d(p,[w,u])+d(u,x)<3\delta+K.$$  The proof of (\ref{4}) is analogous to that of (\ref{3}).
 Figure \ref{fig:thintriangles} illustrates this argument in the special case where $K=0$.

	\begin{figure}[h]
		\centering
		\includegraphics[height=3in]{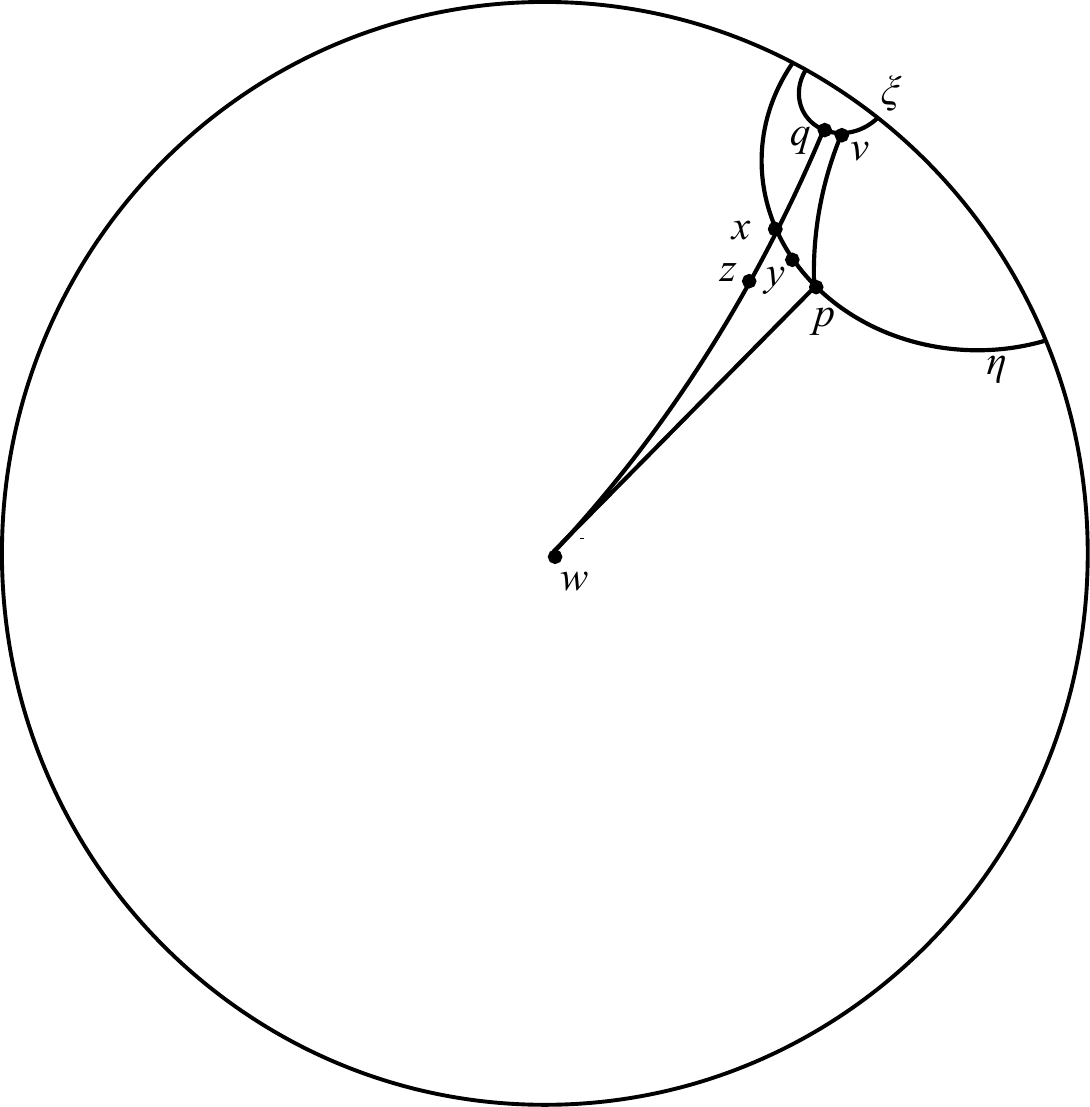}
		\caption{Proof that $d(p,[w,q])<3\delta$ if $\dim(Y)=2$.}
		\label{fig:thintriangles}
	\end{figure}



By (\ref{3}) and (\ref{4}), there exist points $z$ and $z'$ on the segments $[w,q]$ and $[w,q']$ such that 
\begin{equation}\label{5}
d(p,z)< 3\delta+K
\end{equation}
 and 
 \begin{equation}\label{6}
 d(p',z')< 3\delta+K.
 \end{equation}
Let $v$ be the projection of $p$ onto $\xi$, and let $v'$ be the projection of $p'$ onto $\gamma \xi$ (see Figure 2).  Note that $\gamma v$ is the projection of $\gamma p$ onto $\gamma \xi$.  

As the projection of the segment $[p,z]$ to $\xi$ is the segment $[v,q]$, and since projections decrease distances, (\ref{5}) implies

\begin{equation}\label{7}
d(v,q)\leq d(p,z)
\end{equation} 
Analogous arguments show that 
\begin{equation}\label{8}
d(v',q')\leq d(p',z')
\end{equation}
and
\begin{equation}\label{9}
d(\gamma v, v')\leq d(\gamma p, p').
\end{equation}

Use the triangle inequality, (\ref{2}), and (\ref{5})-(\ref{9}) to estimate
\begin{align}\label{10}
		d(\gamma q,q')&\leq d(\gamma q,\gamma v)+d(\gamma v, q')\nonumber\\
		&=d(q,v)+d(\gamma v, q')\nonumber\\
		&\leq d(p,z)+d(\gamma v, q')\\
		&\leq d(p,z) +d(\gamma v, v')+d(v',q')\nonumber\\
		&\leq d(p,z)+d(\gamma p, p')+d(p',z')\nonumber\\
		&<3\delta +K+ 2D+3\delta+K\nonumber\\
		&=2D+2K+6\delta\nonumber
	\end{align}

\begin{figure}[h]
	\centering
	\includegraphics[height=3in]{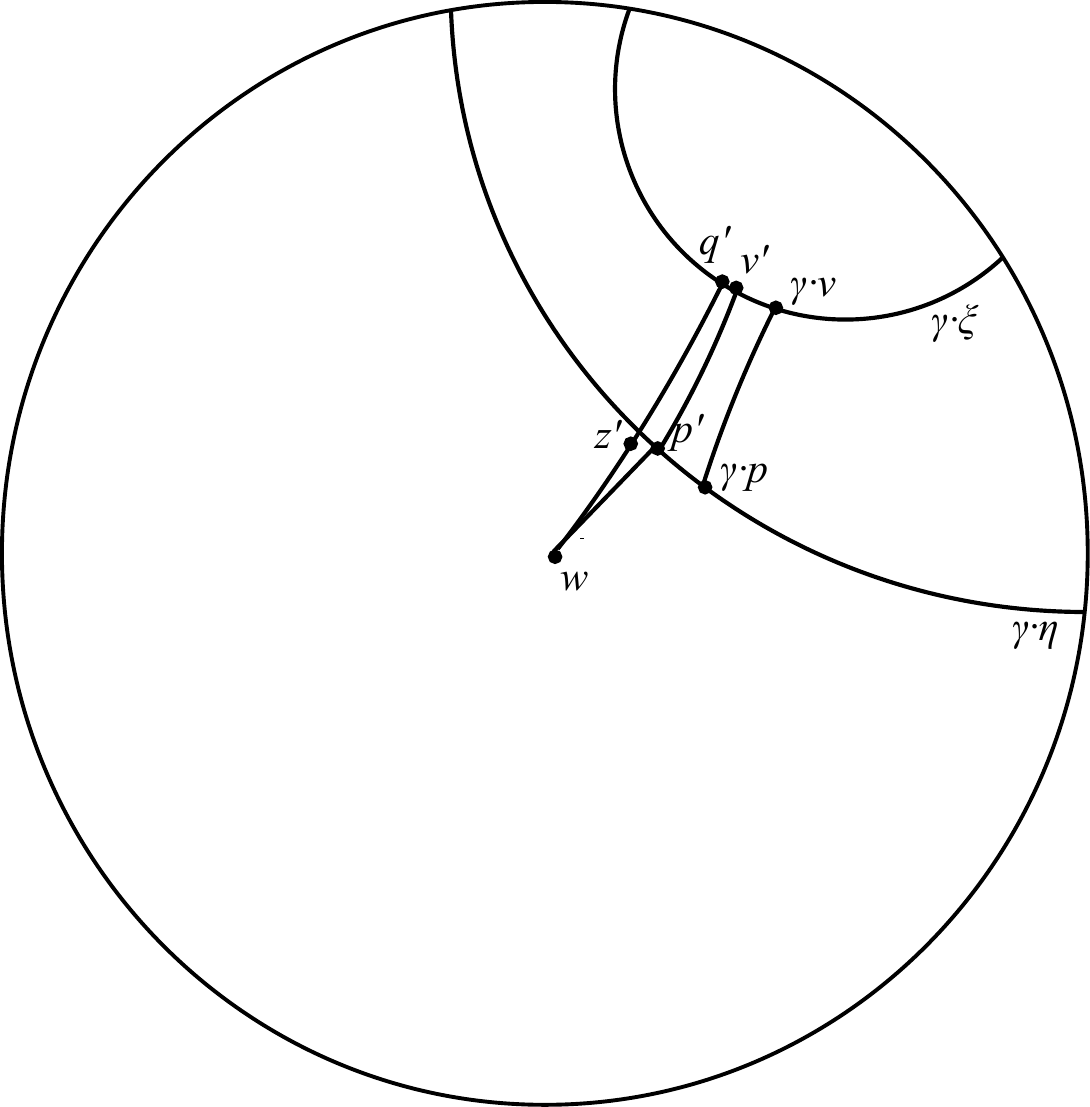}
	\caption{Proof of Proposition \ref{lem:expansion}.}
	\label{fig:proof2}
\end{figure}

The triangle inequality and the assumption $d(w,p)=d(w, \eta)=R$ imply
\begin{equation}\label{11}
d(w,\xi)=d(w,q)\leq d(w,p)+d(p,q)=R+d(p,q).
\end{equation}
The triangle inequality and (\ref{1}) imply
\begin{equation}\label{12}
d(w, \gamma \xi)\leq d(w,\gamma q)\leq d(w,\gamma p)+d(\gamma p, \gamma q)=D+d(p,q).
\end{equation}
The triangle inequality, (\ref{5}), and the assumption $d(w,p)=R$ imply
	\begin{align}\label{13}
		d(w,\xi)&=d(w,q)\nonumber\\
		&=d(w,z)+d(z,q)\nonumber\\
		&\geq d(w,p)-d(p,z) +d(p,q)-d(p,z)\\
		&>R+d(p,q)-6\delta-2K.\nonumber
	\end{align}
Similarly, the triangle inequality, (\ref{2}), (\ref{6}), and (\ref{10}) imply
	\begin{align}\label{14}
		d(w,\gamma \xi)&=d(w,z')+d(z',q')\nonumber\\
		&\geq d(w,p')-d(p',z')+d(p',q')-d(p',z')\nonumber\\
		&> d(p',q')-6\delta-2K\\
		&\geq d(\gamma p,\gamma q)-d(\gamma p,p')-d(\gamma q,q')-6\delta-2K\nonumber\\
		&\geq d(p,q)-2D-d(\gamma q,q')-6\delta-2K\nonumber\\
		&\geq d(p,q)-2D-(2D+2K+6\delta)-6\delta-2K\nonumber\\
		&=d(p,q)-(4D+4K+12\delta).\nonumber
	\end{align}
Combining (\ref{11}) and (\ref{14}) gives $$d(w,\xi)-d(w,\gamma \xi) \leq R+(4D+4K+12\delta)= R+A$$
Combining (\ref{12}) and (\ref{13}) gives $$d(w,\xi)-d(w,\gamma \xi)\geq R-(D+2K+6\delta)> R-A,$$
concluding the proof of the proposition.


\end{proof}

Now consider the visual metric associated to the basepoint $w\in X$ defined by 
	$$d_w(x,y)=e^{-(x|y)_w}$$
for $x,y\in\partial X$. We will prove that $(\Lambda, d_w)$ has Property (ii) from Definition \ref{dfn:qcircle}. Since $d_w$ and $\hat{d}_w$ are bi-Lipschitz equivalent (see Remark \ref{rmk:alternate}) and Property (ii) continues to hold after a change up to bi-Lipschitz equivalence, we will work with $\hat{d}_w$ instead.

Fix a fundamental domain $W$ for the $\Gamma$-action on $Y$ and choose $A\geq 0$ as in Proposition \ref{lem:expansion}.  Let $r_0:=\frac{1}{2}e^{-K}$, $a:=e^{-A}$, and $b:=e^{A}$.

Given $0<r<r_0$ and $N$ a ball of radius $r$ in $\Lambda$, let $\eta$ be the geodesic connecting the endpoints of $N$ so that $2r=\diam(N)=e^{-d(w,\eta)}$. Note that since $2r<e^{-K}$, we have $d(w,\eta)\geq-\log(2r)-K>0$ so that $N_K(\eta)$ does not contain $w$. Further since $r<\frac{1}{2}$, we know that $\diam(N)<\diam(\partial X\backslash N)$, so that $N$ lies on the side of $\eta$ not containing $w$.

Let $p$ be the projection of $w$ onto $\eta$ and choose $\gamma$ such that $\gamma p\in W$. Now let $x,y\in N$ be distinct and let $\xi$ be the geodesic joining $x$ and $y$. By Proposition \ref{lem:expansion}, we have
	$$R-A\leq d(w,\xi)-d(w,\gamma\cdot \xi)\leq R+A,$$
where $R:=d(w,\eta)=-\log(r).$ Further note that
	$$\frac{\hat{d}_w(\gamma x,\gamma y)}{\hat{d}_w(x,y)}=e^{d(w,\xi)-d(w,\gamma\cdot \xi)}$$
so we have
	$$e^{-A} e^{R}\leq \frac{\hat{d}_w(\gamma x,\gamma y)}{\hat{d}_w(x,y)}\leq e^{A} e^{R},$$
or equivalently $$\frac{a}{r}\leq \frac{\hat{d}_w(\gamma x,\gamma y)}{\hat{d}_w(x,y)}\leq \frac{b}{r}.$$
This establishes property (ii) in Definition \ref{dfn:qcircle}, and hence completes the proof of Theorem \ref{thm:main}.




\section{Constructing exotic almost-isometries} \label{sec:main2}

This section is devoted to the proof of Theorem \ref{thm:main2}. In view of Corollary \ref{cor:lipequiv}, we want to produce
a $k$-dimensional family $\mathcal F_k$ of equal area metrics on a higher genus surface $M$, 
which all have the same topological entropy, but whose lifts to the universal cover are not isometric to each other. 

\subsection{Perturbations of metrics}

We start with a fixed reference hyperbolic metric $g_0$ on $M$, normalized to have constant curvature $-2$. Pick
$(k+2)$ distinct points $p_1, \ldots p_{k+2} \in M$, and
choose $r_2$ smaller than the injectivity radius of $M$ and satisfying $2r_2< \inf_{i\neq j} \{d(p_i, p_j)\}$. 
Let $U_i$ denote the open metric ball of radius $r_2$ centered at $p_i$ -- note that the $U_i$ are all isometric to each 
other, and are pairwise disjoint. Now choose $r_1<r_2$ so 
that the area of the ball of radius $r_1$ is at least $4/5$ the area of the ball of radius $r_2$. Denote by 
$V_i \subset U_i$ the ball of radius $r_1$ centered at each $p_i$. 

We will vary the metric $g_0$ by introducing a perturbation on each of the $U_i$ in the following manner. Let us choose a 
smooth bump function $\rho: [0, \infty) \rightarrow [0,1]$ with the property that $\rho|_{[0, r_1]} \equiv 1$ and 
$\rho|_{[r_2, \infty)}\equiv 0$. Next define $u_i: M \rightarrow [0,1]$ via $u_i(x) := \rho\left( d(x, p_i)\right)$. Given
a parameter $\vec t :=(t_1, \ldots , t_{k+2}) \in \mathbb R ^{k+2}$, define the function 
$u_{\vec t}: M\rightarrow [0, \infty)$ by setting $u_{\vec t}:= t_1u_1 +\cdots + t_{k+2} u_{k+2}$. Finally,
we define the metric $g_{\vec t} := e^{2u_{\vec t}}g_0$ (and note the identification $g_{\vec 0}=g_0$). 
The family $\mathcal F_k$ will be obtained by choosing suitable values of $\vec t$ close to $\vec 0$.

Since the metric $g_{\vec t}$ is obtained by making a conformal change on each $U_i$, 
and since the $U_i$ are pairwise disjoint, we first analyze the behavior of such a change on an individual $U_i$. To 
simplify notation, denote by $V\subset U$ open balls of radius $r_1< r_2$ centered at a point $p$ 
in the hyperbolic plane $\mathbb H^2_{-2}$ of 
curvature $-2$, and set $g_t := e^{2t u}g_0$ where $g_0$ is the hyperbolic metric of curvature $-2$, and
$u: \mathbb H^2_{-2} \rightarrow [0, \infty)$ is given by $u(x) := \rho\left( d(p, x) \right)$.
We start with the easy:

\begin{lemnr}
As $t \to 0$, we have the following estimates:
\begin{enumerate}
\item the curvatures $K(g_t)$ tend uniformly to $-2$.
\item the area $Area(U; g_t)$ of the ball $U$ tends to $Area(U; g_0)$.
\item the area $Area(V; g_t)$ of the ball $V$ tends to $Area(V; g_0)$.
\end{enumerate}
\label{lem:small-lambda}
\end{lemnr}

\begin{proof}
This is straightforward from the formulas expressing how curvature and area change when one makes a conformal change 
of metric. We have that the new curvature $K(g_t)$ is related to the old curvature $K(g_0)$ via the formula
$$K(g_t) =  (e^{-2u})^t K(g_0) - t (e^{-2 u})^t \Delta u$$
where $\Delta u$ denotes the Laplacian of the function $u$ in the hyperbolic metric $g_0$. As $t$ tends to zero,
it is clear that the expression to the right converges to $K(g_0)$ uniformly, giving (1). 
Similarly, the area form $dg_t$ for the new metric is related to the area form $dg_0$ for the original metric via the 
formula $dg_t = (e^{2u})^t dg_0$ giving us (2) and (3).
\end{proof}


\subsection{Lifted metrics are almost-isometric}

Next we establish that, for suitable choices of the parameter $\vec t$, we can arrange for the lifted metrics to be
almost-isometric. By Lemma \ref{lem:small-lambda}, we can take the parameters $\vec t$ close enough to $\vec 0$
to ensure that all the metrics we consider have sectional curvatures $\leq -1$. 
Then from Corollary \ref{cor:lipequiv}, it suffices to consider values of the parameter $\vec t$ for
which the corresponding metrics have the same topological entropy for the geodesic flow on $T^1M$. 
Notice that varying $\vec t$ near $\vec 0$ gives a $C^\infty$ family of 
perturbations of the metric $g_{0}$. Work of Katok, Knieper, Pollicott and Weiss \cite[Theorem 2]{KKPW} then implies that the
topological entropy map $h$, when restricted to any line $l(s)$ through the origin $\vec 0$ in the $\vec t$-space, 
is a $C^\infty$ map. Moreover the derivative of $h$ along the line is given by (see \cite[Theorem 3]{KKW})
$$\frac{\partial}{\partial s}\Big|_{s=0}h\left(g_{l(s)}\right) = -\frac{h(g_0)}{2} \int_{T^1M}\frac{\partial}{\partial s}\Big|_{s=0}g_{l(s)}(v,v) d\mu_0$$
where $T^1M$ denotes the unit tangent bundle of $M$ with respect to the $g_0$-metric, and $\mu_0$ denotes the
Margulis measure of $g_0$ (the unique measure of maximal entropy for the $g_0$-geodesic flow on $T^1M$). 

Consider the map $F: \mathbb R^{k+2} \rightarrow \mathbb R$ given by $F(t_1, \ldots , t_{k+2}):= h(g_{(t_1,\ldots , t_{k+2})})$, 
where $h$ denotes the topological entropy of (the geodesic flow associated to) a metric. Let us compute the directional derivative
in the direction $\frac{\partial}{ \partial t_1}$:
\begin{align*}
\frac{\partial F}{\partial t_1} (0,\dots, 0) &= \frac{d}{dt}\Big|_{t=0}h\left(g_{(t, 0,\ldots, 0)}\right) = -\frac{h(g_0)}{2} \int_{T^1M}\frac{d}{d t}\Big|_{t=0}g_{(t,0,\ldots, 0)}(v,v) d\mu_0 \\
& = -\frac{h(g_0)}{2} \int_{T^1M}\frac{d}{d t}\Big|_{t=0} e^{2tu_1\big(\pi(v)\big)} d\mu_0\\
& = -\frac{h(g_0)}{2} \int_{T^1M}2u_1\big(\pi(v)\big) d\mu_0
\end{align*}
where $\pi: T^1M \rightarrow M$ is the projection from the unit tangent bundle onto the surface $M$. Finally, we observe
that by construction $u_1$ is a non-negative function, which is identically zero on the complement of $U_1$, and identically 
one on the set $V_1$. Hence the integral above is positive, and we obtain $\frac{\partial F}{\partial t_1}(\vec 0)<0$. 

Now, a similar calculation applied to each of the other coordinates gives us the general formula for the
directional derivative of $F$. The gradient of $F$ is given by the non-vanishing vector:
$$\nabla F = -{h(g_0)}\int _{T^1M}\langle u_1\big(\pi(v)\big), \ldots , u_{k+2}\big(\pi(v)\big)\rangle  d\mu_0.$$ 
In fact, since each $u_i$ is supported solely on $U_i$, and each $u_i$ is defined as $\rho\left(d(p_i, x)\right)$ on 
the $U_i$, each of the integrals in the expression for $\nabla F$ has the same value. So $\nabla F$ is just a nonzero multiple
of the vector $\langle 1, \ldots, 1\rangle$.

The implicit function theorem now locally gives us an embedded codimension one submanifold $\sigma(z)$ 
(where $z \in \mathbb{R}^{k+1}$, $||z|| < \epsilon$) in the $(t_1,\ldots, t_{k+2})$-space, with normal vector 
$\langle 1, \ldots , 1\rangle$ at the point $\sigma(\vec 0)=\vec 0$, on which the topological entropy functional is constant.  From Corollary \ref{cor:lipequiv}, we see that the lifts of these metrics to the universal cover are all pairwise almost-isometric.

\subsection{Lifted metrics are not isometric}


\begin{lemnr}
There is an $\epsilon>0$ so that if the parameters $\vec s = (s_1, \ldots , s_{k+2})$ and $\vec t=(t_1, \ldots , t_{k+2})$ satisfy 
$0<|s_i|< \epsilon$ and $0<|t_i|<\epsilon$ and the lifted metrics $(\tilde M, \tilde g_{\vec s})$ 
and $(\tilde M, \tilde g_{\vec t})$ are isometric to each other, then we must have an equality of multisets $\{s_1, \ldots, s_{k+2}\} =\{t_1, \ldots , t_{k+2}\}$.
\label{lem:multiset}\end{lemnr}

We recall that a multi-set is a set with multiplicities associated to each element. Equality of multi-sets means not only that the
underlying sets are equal, but that the corresponding multiplicities are equal.

\begin{proof}
By Lemma \ref{lem:small-lambda}, it is possible to pick $\epsilon$ small enough so that, for all parameters $\vec s, \vec t$ within the $\epsilon$-ball around $\vec 0$, we
have that 
	$$\Area(V_i;g_{\vec t})\geq \frac{3}{4}\Area(U_j;g_{\vec s})$$
for every $1\leq i,j\leq k+2$.

Now let us assume that there is an isometry $\Phi: (\tilde M, \tilde g_{\vec s})
\rightarrow (\tilde M, \tilde g_{\vec t})$. Observe that the lifted metrics have the following properties:
\begin{enumerate}[(i)]
\item on the complement of the lifts of the $U_i$, both metrics have curvature identically $-2$.
\item on any lift of the set $V_1$, the metric $\tilde g_{\vec s}$ has curvature identically
$-2e^{-2s_1}$.
\item on any lift of the set $V_i$, the metric $\tilde g_{\vec t}$ has curvature identically $-2e^{-2t_i}$.
\end{enumerate}

Take a lift $\tilde V_1$ of $V_1$ in the source, and consider its image
under $\Phi$. The metric in the source has curvature identically $-2e^{-2s_1}$ on this lift $\tilde V_1$, and
since $\Phi$ is an isometry, the image set $\Phi(\tilde V_1)$ must have the same curvature. From property
(i), we see that $\Phi(\tilde V_1)$ must lie, as a set, inside the union of lifts of the $U_i$. Since $\Phi(\tilde V_1)$
is path-connected, it must lie inside a single connected lift $\tilde U_i$ of one of the $U_i$. But from the area 
estimate, we see that for the $\tilde V_i \subset \tilde U_i$ inside the lift, one has that the intersection 
$\Phi(\tilde V_1) \cap \tilde V_i$ is non-empty. Looking at the curvature of a point in the intersection, we see that
$$-2e^{-2s_1} = -2e^{-2t_i}$$
and hence that $s_1 = t_i$ for some $i$. Applying the same argument to each of $s_i, t_i$ completes the proof.
\end{proof}

Now pick a vector $\vec v = \langle v_1, \ldots , v_{k+2} \rangle$ with the property that $v_1 + \cdots + v_{k+2} = 0$,  
and such that $v_i \neq v_j$ for each $i\neq j$. Notice that the first constraint just means that $\vec v \cdot \nabla F = 0$, and hence that
$\vec v$ is tangent to the $(k+1)$-dimensional submanifold $\sigma$. So there exists a curve $\gamma \subset \sigma$
satisfying $\gamma(0) = \vec 0$, and $\gamma^\prime (0) = \vec v$. Notice that, from our second condition, when 
$t\approx 0$ we have $\gamma(t) \approx (v_1 t, \ldots , v_{k+2}t)$, and hence the point $\gamma(t)$ 
has {\it all coordinates distinct}. It follows from Lemma \ref{lem:multiset} that, for any $t\approx 0$ ($t\neq 0$), 
one can find a small enough connected neighborhood $W_t$ of $\gamma(t)$ with the property that all the metrics in that neighborhood have lifts to the universal cover that are pairwise non-isometric.

\subsection{Metrics with equal area}

Now consider the smooth function $$A:\sigma \rightarrow \mathbb{R}$$ defined by $A(z):=\Area(g_{\sigma(z)})$ for each $z \in \sigma$.    The change of area formula for a conformal change of metric (see the proof of Lemma \ref{lem:small-lambda}) implies that $A$ is nonconstant on $W_t$.  By Sard's theorem, there is a regular value $r$ of $A$ in the interval $A(W_t)$.  Then $\tau := A^{-1}(r)$ is a smooth $k$-dimensional
submanifold of the $(k+1)$-dimensional manifold $\sigma$ consisting of parameters for area $r$ metrics.  A connected component $\mathcal F_k$ of $W_t\cap A^{-1}(r)$ satisfies all of the constraints of Theorem \ref{thm:main2}.

\section{Concluding remarks}\label{sec:conclusion}

As the reader undoubtedly noticed, our results rely heavily on the surprising result of Falconer and Marsh. As such,
it is very specific to the case of circle boundaries -- which essentially restricts us to surface 
groups (see Gabai \cite{G}). In higher dimensions, we would not expect the bi-Lipschitz class
of a self-similar metric on a sphere to be classified by its Hausdorff dimension. Thus, the following problem seems 
substantially more difficult.

\begin{conj}
Let $M$ be a smooth closed manifold of dimension $\geq 3$, and assume that $M$ supports a negatively curved
Riemannian metric. Then $M$ supports a pair of equal volume Riemannian metrics $g_1, g_2$ with curvatures
$\leq -1$, and having the property that
the Riemannian universal covers $(\tilde M, \tilde g_i)$ are almost-isometric, but are {\bf not} isometric.
\end{conj}

In another direction, if one were to {\it drop} the dimension, then there are many examples of $0$-dimensional spaces
having analogous self-similarity properties (i.e. properties (ii), (iii) in Definition \ref{dfn:qcircle}). The metrics on these
boundaries turn them into Cantor sets -- and the classification of (metric) Cantor sets up to bi-Lipschitz equivalence 
seems much more complex than in the circle case (for some foundational results on this problem, see for instance Falconer 
and Marsh \cite{FM} and Cooper and Pignatoro \cite{CP}). Of course, from the viewpoint of boundaries, such spaces
would typically arise as the boundary at infinity of a metric tree $T$. This suggests the following:

\begin{prob}
Study periodic metrics on trees up to the relation of almost-isometry.
\end{prob}

In particular, invariance of the metric under a cocompact group action translates to additional constraints on the 
canonical visual metric on $\partial T$, e.g. the existence of a large (convergence) group action via conformal automorphisms 
(compare with the main theorem in \cite{CP}). It would be interesting to see if this makes the bi-Lipschitz classification 
problem any easier.

\vskip 5pt

Finally, given a pair of quasi-isometric spaces, we can consider the collection of {\it all} quasi-isometries between them, and try to 
find the quasi-isometry which has smallest multiplicative constant. More precisely, given a pair of quasi-isometric metric 
spaces $X_1, X_2$, define the real number $\mu(X_1, X_2)$ to be the infimum of the real numbers $C$ with the property
that there exists some $(C,K)$-quasi-isometry from $X_1$ to $X_2$. We can now formulate the:

\begin{prob}
Given a pair of quasi-isometric metric spaces $X_1, X_2$, can one estimate $\mu(X_1, X_2)$? Can one find a 
$(C,K)$-quasi-isometry from $X_1$ to $X_2$, where $C= \mu(X_1, X_2)$? In particular, can one find a pair of 
quasi-isometric spaces $X_1, X_2$
which are  {\bf not} almost-isometric, but which nevertheless satisfy $\mu(X_1, X_2)=1$?
\end{prob}

Our Corollary \ref{cor:gap} gives a complete answer in the case where the $X_i$ are universal covers of locally CAT(-1)
metrics on surfaces -- the real number $\mu(X_1, X_2)$ is exactly the ratio of the Hausdorff dimensions of the canonical
visual metrics on the boundary, and one can always find a quasi-isometry with multiplicative constant $\mu(X_1, X_2)$. 
It is unclear what to expect in the more general setting.

\end{document}